\def\timestamp{%
Time-stamp: <RF-unctble.tex: vrijdag 24-03-2023 at 20:27:57 (cet)>}
\def\stripname Time-stamp: <#1: #2 #3 at #4 #5>{#3/#4 (#1)}
\edef\filedate{\expandafter\stripname\timestamp}
\documentclass[a4paper]{amsart}

\DeclareMathSymbol\le \mathrel{AMSa}{"36}
\DeclareMathSymbol\ge \mathrel{AMSa}{"3E}
\newcommand\orpr[2]{\langle{#1},{#2}\rangle}
\newcommand\card[1]{\mathopen|{#1}\mathclose|}
\newcommand\bcard[1]{\bigl|{#1}\bigr|}
\newcommand\cl{\overline}
\newcommand\tkp{(2^\kappa)^+}
\newcommand\RF{<_\mathrm{RF}}
\newcommand\RFe{\le_\mathrm{RF}}
\newcommand\dom{\operatorname{dom}}

\newcommand\cee{\mathfrak{c}}
\newcommand\gothA{\mathfrak{A}}
\newcommand\gothU{\mathfrak{U}}

\newcommand\calF{\mathcal{F}}
\newcommand\calG{\mathcal{G}}

\newtheorem{lemma}{Lemma}[section]

\usepackage{amsrefs}

\begin{document}
\title[Long chains in the Rudin-Frol\'ik order]%
  {Long chains in the Rudin-Frol\'ik order for uncountable cardinals}
\author[K. P. Hart]{Klaas Pieter Hart}
\address{Faculty EEMCS\\TU Delft\\
         Postbus 5031\\2600~GA {} Delft\\the Netherlands}
\email{k.p.hart@tudelft.nl}
\urladdr{https://fa.ewi.tudelft.nl/\~{}hart}

\begin{abstract}
We point out that a construction by Butkovi\v{c}ov\'a of a chain
of length~$\cee^+$ in the Rudin-Frol\'ik order on~$\beta\omega$
can easily be adapted to produce, given an uncountable cardinal~$\kappa$,
a chain of length~$\tkp$ in the Rudin-Frol\'ik order on~$\beta\kappa$.
\end{abstract}


\date{\filedate}

\maketitle

\section*{Introduction}

The Rudin-Frol\'ik order~$\RFe$ of ultrafilters was defined by Frol\'ik 
in~\cite{MR203676} and used to prove the non-homogeneity 
of~$\beta\omega_0\setminus\omega_0$.
The order is tree-like on $\beta\omega_0\setminus\omega_0$;
the predecessors of an element are linearly ordered and every point has
$2^\cee$ successors.

Many natural questions about this order have been answered for
$\beta\omega_0\setminus\omega_0$.
For example, every point has at most $\cee$~many predecessors, 
hence a chain can have cardinality at most~$\cee^+$ and in~\cite{MR730151} 
Butkovi\v{c}ov\'a constructed a chain of that maximum possible order type.

The definition of $\RFe$ was give by Frol\'ik for ultrafilters on arbitrary 
sets and its stand to reason that one asks whether the results for
ultrafilters on~$\omega_0$ can be generalized to uncountable cardinals.

It is the purpose of this note to point out that Butkovi\v{c}ov\'a's 
construction can be used almost verbatim to produce in $\beta\kappa$, 
where $\kappa$~is uncountable, an $\RFe$-chain of length~$(2^\kappa)^+$.

\section{Preliminaries}

\subsection{The Rudin-Frol\'ik order}

Let $\kappa$ be an infinite cardinal.
An indexed set $\{u_\alpha:\alpha\in\kappa\}$ is said to be $\kappa$-discrete
if there is a partition $\{U_\alpha:\alpha\in\kappa\}$ into sets of 
cardinality~$\kappa$ such that $U_\alpha\in u_\alpha$ for all~$\alpha$.

In case $\kappa=\omega_0$ this notion coincides with relative discreteness
of the set in the \v{C}ech-Stone compactification $\beta\omega_0$ because
in Hausdorff spaces one can separate the elements of a countable relatively
discrete set can be separated by pairwise disjoint open sets.

If $X=\{u_\alpha:\alpha\in\kappa\}$ is $\kappa$-discrete then its closure 
in~$\beta\kappa$ is homeomorphic to~$\beta\kappa$; 
the map $f:\alpha\to\beta\kappa$ given by $f(\alpha)=u_\alpha$ induces
a homeomorphism between $\beta\kappa$ and the closure 
of $\{u_\alpha:\alpha\in\kappa\}$.
If $u\in\beta\kappa$ then, following Frol\'ik, one usually denotes 
$\beta f(u)$ by $\sum(X,u)$.
 
The Rudin-Frol\'ik order~$\RFe$ on $\beta\kappa$ is defined by $u\RFe v$ iff
there is a $\kappa$-discrete set~$X$ such that $v=\sum(X,u)$.
The proof in~\cite{MR273581} that $\RFe$~is a partial order of the types
of ultrafilters on~$\omega_0$ is readily addapted to other cardinals.

\subsection{Stratified sets of filters}

The main tool in the construction is that of a stratified set of filters.

Let $\kappa$ be an infinite cardinal.
A stratified set of filters on~$\kappa$ is an $\alpha\times\kappa$-matrix
$\langle \calF_{\beta,\eta}:\orpr\beta\eta\in\alpha\times\kappa\rangle$
of filters on~$\kappa$, where $\alpha$~is an ordinal, such that
\begin{itemize}
\item there is a choice $F_{\beta,\eta}\in\calF_{\beta,\eta}$ of elements
      such that for every $\beta\in\alpha$ the family
      $\{F_{\beta,\eta}:\eta\in\kappa\}$ is pairwise disjoint, and
\item if $\beta<\delta<\alpha$, $\eta\in\kappa$ and $F\in\calF_{\beta,\eta}$
      then $\{\zeta\in\kappa:F\in\calF_{\delta,\zeta}\}$ has cardinality~$\kappa$. 
\end{itemize}

Note that the filters in a stratified set are uniform, except possibly
for those in the last row, if $\alpha$~is a successor.
For if $F\in\calF_{\beta,\eta}$ and $\beta+1<\alpha$ then $F$~intersects
$\kappa$~many of the pairwise disjoint sets $F_{\beta+1,\zeta}$, so it has
cardinality~$\kappa$.

In our construction it will always be the case that if $\alpha=\beta+1$
the filters~$\calF_{\beta,\zeta}$ are all uniform. 

\subsection{Independent matrix}\label{subsec.matrix}

A secondary tool in the construction is that of an independent matrix of sets,
that is, a matrix
$\langle A_{\xi,\eta}:\orpr\xi\eta\in2^\kappa\times\kappa\rangle$
of subsets of~$\kappa$ such that
\begin{enumerate}
\item for each $\xi$ the family $\{A_{\xi,\eta}:\eta\in\kappa\}$
      is a partition of~$\kappa$ into sets of cardinality~$\kappa$, and
\item for each finite function~$p\subset2^\kappa\times\kappa$ the intersection
      $\bigcap_{\xi\in\dom p}A_{\xi,p(\xi)}$ has cardinality~$\kappa$.
\end{enumerate}
For a construction of such a family see~\cite{MR196693}*{Theorem~3}.


\section{The main result}

As in~\cite{MR730151} we construct a triangular array of $\kappa$-discrete
families of ultrafilters.
That is, an array
$$
\bigl<\gothU_{\alpha,\beta}:\beta<\alpha<\tkp\bigr>
$$
where each $\gothU_{\alpha,\beta}$ is an indexed $\kappa$-discrete 
family $\langle u(\alpha,\beta,\eta):\eta\in\kappa\rangle$
of ultrafilters on~$\kappa$. 

The demands on this array are that for every $\alpha<\tkp$
\begin{enumerate}
\item the row $\langle\gothU_{\alpha,\beta}:\beta<\alpha\rangle$
      is a stratified family of ultrafilters, and\label{cond.1}
\item if $\beta<\delta<\alpha$ and $\eta\in\kappa$ then\label{cond.2}
      $$
      u(\alpha,\beta,\eta)=
         \sum\bigl(\gothU_{\alpha,\delta},u(\delta,\beta,\eta)\bigr)
      $$ 
\end{enumerate}

This will yield very many chains of length~$\tkp$ in the Rudin-Frol\'ik
order on~$\beta\kappa$.
Indeed, for each fixed pair $\orpr\beta\eta$ condition~\eqref{cond.2} above
shows that $u(\delta,\beta,\eta)\RF u(\alpha,\beta,\eta)$ whenever
$\beta<\delta<\alpha$ so that the sequence
$$
\bigl<u(\alpha,\beta,\eta):\beta<\alpha<\tkp\bigr>
$$
is $\RF$-increasing.

\subsection*{The construction}

We construct our array recursively, one row at a time.
At the same time we construct an array 
$$
\bigl<\gothA_{\alpha,\beta}:\beta<\alpha<\tkp\bigr>
$$
of partitions, where 
$\gothA_{\alpha,\beta}=\{U(\alpha,\beta,\eta):\eta\in\kappa\}$ is a row
in the independent matrix 
$\langle A_{\xi,\eta}:\orpr\xi\eta\in2^\kappa\times\kappa\rangle$
from subsection~\ref{subsec.matrix}.
We will always have $U(\alpha,\beta,\eta)\in u(\alpha,\beta,\eta)$,
so $\gothA_{\alpha,\beta}$ witnesses the $\kappa$-discreteness 
of~$\gothU_{\alpha,\beta}$.

To begin the construction we take the first row $\{A_{0,\eta}:\eta\in\kappa\}$
from our matrix, so $U(1,0,\eta)=A_{0,\eta}$, and we choose, 
for each $\eta$, a uniform ultrafilter $u(1,0,\eta)$ such that 
$U(1,0,\eta)\in u(1,0,\eta)$.
Then the set $\gothU_{1,0}=\{u(1,0,\eta):\eta\in\kappa\}$ is $\kappa$-discrete,
and this one-element row is stratified, vacuously.

\bigskip
Now assume that $\alpha\in\tkp$ is given and that we have an array
$$
\langle\gothU_{\beta,\gamma}:\gamma<\beta<\alpha\rangle
$$
that meets requirements~\eqref{cond.1} and~\eqref{cond.2} up to~$\alpha$.
So each row $\langle\gothU_{\beta,\gamma}:\gamma<\beta\rangle$ is stratified
and we have
$$
u(\beta,\gamma,\eta)=
         \sum\bigl(\gothU_{\beta,\delta},u(\delta,\gamma,\eta)\bigr)
$$ 
whenever $\gamma<\delta<\beta$ and $\eta\in\kappa$.

\smallskip
We construct the $\alpha$th row.

Since $\alpha<\tkp$ we can take an injective map $i:\alpha\to2^\kappa$
and make an independent matrix
$\langle U(\alpha,\beta,\eta):\orpr\beta\eta\in\alpha\times\kappa\rangle$ 
by setting $U(\alpha,\beta,\eta)=A_{i(\xi),\eta}$ for all~$\xi$ and~$\eta$.
This then also gives us the partitions~$\gothA_{\alpha,\beta}$ 
for~$\beta<\alpha$.

\bigskip
Using this matrix, and the ultrafilters constructed thus far, we construct
a stratified family
$$
\langle\calF_{\beta,\eta}:\orpr\beta\eta\in\alpha\times\kappa\rangle
$$
of filters, as follows.

Let $\orpr\beta\eta\in\alpha\times\kappa$.
We let $\calF_{\beta,\eta}$ be the filter generated by the union of the 
following families:
\begin{enumerate}
\item the Fr\'echet filter 
      $\{F\subseteq\kappa:\card{\kappa\setminus X}<\kappa\}$,
\item for $\gamma<\beta$ the singleton set $\{U(\alpha,\gamma,\zeta)\}$, 
      where $\zeta$~is such that $\eta\in U(\beta,\gamma,\zeta)$,
\item the singleton set $\{U(\alpha,\beta,\eta)\}$,
\item for $\gamma\in(\beta,\alpha)$ the family
      $\bigl\{\bigcup_{\zeta\in U}U(\alpha,\gamma,\zeta):
           U\in u(\gamma,\beta,\eta)\bigr\}$
\end{enumerate}
Using the fact that 
$\bigl<U(\alpha,\gamma,\zeta):\orpr\gamma\zeta\in\alpha\times\kappa\bigr>$
is an independent matrix one readily checks that the union of the families
given above has the finite intersection property and that
$\calF_{\beta,\eta}$~is indeed a uniform filter.
It remains to show that the resulting family is stratified.

For this we use that the formulas above are used at every stage of 
the construction and hence that (1)--(4) hold for all triples 
$\gamma<\beta<\delta$ of ordinals below~$\alpha$.

\medbreak
So let $\orpr\beta\eta$ and $\delta\in(\beta,\alpha)$ be given.
We calculate for every element~$G$ of the generating family 
of~$\calF_{\beta,\eta}$ the set $X_G=\{\zeta:G\in\calF_{\delta,\zeta}\}$
and show that it belongs to~$u(\delta,\beta,\eta)$.
This will show that $X_G\in u(\delta,\beta,\eta)$ for all~$G$ 
in~$\calF_{\beta,\eta}$, and hence that all these sets have cardinality~$\kappa$.

\begin{enumerate}
\item If $G$ belongs to the Fr\'echet filter then $X_G=\kappa$ because 
      all filters are uniform.

\item If $G=U(\alpha,\gamma,\xi)$, with $\gamma<\beta$, then 
      $G\in\calF_{\delta,\zeta}$ iff 
      $\zeta\in U(\delta,\gamma,\xi)$ and so
      $X_G=U(\delta,\gamma,\xi)$.
      But since $U(\alpha,\gamma,\xi)\in\calF_{\beta,\eta}$ we also have
      $\eta\in U(\beta,\gamma,\xi)$, which means that during the construction 
      of row~$\delta$ we ensured via~(2) 
      that 
      $U(\delta,\gamma,\xi)\in u(\delta,\beta,\eta)$, that is, 
      $X_G\in u(\delta,\beta,\eta)$.

\item If $G=U(\alpha,\beta,\eta)$ then $G\in\calF_{\delta,\zeta}$ iff
      $\zeta\in U(\delta,\beta,\eta)$, 
      so $X_G=U(\delta,\beta,\eta)$, and $X_G\in u(\delta,\beta,\eta)$.

\item If $G=\bigcup_{\xi\in U}U(\alpha,\gamma,\xi)$ with $\beta<\gamma<\delta$
      and $U\in u(\gamma,\beta,\eta)$,  
      then $G\in\calF_{\delta,\zeta}$ iff $\zeta\in U(\delta,\gamma,\xi)$ for
      some~$\xi$ in~$U$, and so $X_G=\bigcup_{\xi\in U}U(\delta,\gamma,\xi)$
      and so
      $X_G\in\sum\bigl(\gothU_{\delta,\gamma},u(\gamma,\beta,\eta)\bigr)
           =u(\delta,\beta,\eta)$.
\item If $G=\bigcup_{\xi\in U}U(\alpha,\delta,\xi)$ 
      for some $U\in u(\delta,\beta,\eta)$
      then $X_G=U$ because if $\zeta\in\kappa$ then 
      $G\in\calF_{\delta,\zeta}$ iff 
      $G\supseteq U(\alpha,\delta,\zeta)$ iff $\zeta\in U$.
\item If $G=\bigcup_{\xi\in U}U(\alpha,\gamma,\xi)$ with $\delta<\gamma$
      and $U\in u(\gamma,\beta,\eta)$ then, because  
      $u(\gamma,\beta,\eta)=
        \sum\bigl(\gothU_{\gamma,\delta},u(\delta,\beta,\eta)\bigr)$,
      we have $\{\zeta:U\in u(\gamma,\delta,\zeta)\}\in u(\delta,\beta,\eta)$.
      But the definition of the~$\calF_{\delta,\zeta}$ then implies
      that $X_G=\{\zeta:U\in u(\gamma,\delta,\zeta)\}$ and 
      so $X_G\in u(\delta,\beta,\eta)$.
\end{enumerate}

In the next section we will show how to find a stratified set
$\langle u(\alpha,\beta,\eta):\orpr\beta\eta\in\alpha\times\kappa\rangle$
of ultrafilters such that $\calF_{\beta,\eta}\subseteq u(\alpha,\beta,\eta)$
for all~$\orpr\beta\eta$.

We now show that that such a stratified set will satisfy
$$
      u(\alpha,\beta,\eta)=
         \sum\bigl(\gothU_{\alpha,\delta},u(\delta,\beta,\eta)\bigr)
$$ 
whenever $\beta<\delta<\alpha$ and $\eta\in\kappa$.

Let $V\in u(\alpha,\beta,\eta)$ and $I=\{\zeta:V\in u(\alpha,\delta,\zeta)\}$.
We must show that $I\in u(\delta,\beta,\eta)$.

Let $U\in u(\delta,\beta,\eta)$, 
then $W=\bigcup_{\xi\in U} U(\alpha,\delta,\xi)$ belongs to~$\calF_{\beta,\eta}$
and so $V\cap W\in u(\delta,\beta,\eta)$, which then implies
that $J=\{\zeta:V\cap W\in u(\alpha,\delta,\zeta)\}$ has cardinality~$\kappa$.
But $J\subseteq I\cap U$, so that $I\cap U$ has cardinality~$\kappa$.
As $U$~was arbitrary this shows that $I\in u(\delta,\beta,\eta)$.

\section{From stratified sets of filters to stratified sets of ultrafilters}

We are given a stratified set 
$\langle \calF_{\beta,\eta}:\orpr\beta\eta\in\alpha\times\kappa\rangle$
of filters and we must construct a stratified set 
$\langle u_{\beta,\eta}:\orpr\beta\eta\in\alpha\times\kappa\rangle$
of ultrafilters such that $\calF_{\beta,\eta}\subseteq u_{\beta,\eta}$
for all~$\orpr\beta\eta$.

The following lemma gives us the successor step in the construction below.

\begin{lemma}\label{lemma.refining}
Let
$\langle \calF_{\beta,\eta}:\orpr\beta\eta\in\alpha\times\kappa\rangle$
be a stratified set of filters on~$\kappa$ and let $X\subseteq\kappa$.
Then there is a stratified set
$\langle\calG_{\beta,\eta}:\orpr\beta\eta\in\alpha\times\kappa\rangle$
of filters such that for all $\orpr\beta\eta$ we have
$\calF_{\beta,\eta}\subseteq\calG_{\beta,\eta}$, and
$X\in\calG_{\beta,\eta}$ or $\kappa\setminus X\in\calG_{\beta,\eta}$.
\end{lemma}

\begin{proof}
We start by deciding to which filters we add~$X$ and to which we add its 
complement.

We let $C=\{\orpr\beta\eta:X\in\calF_{\beta,\eta}\}$ and close it off in a
certain way:
define $C_0=C$, and, recursively,
$C_\xi=\bigcup_{\delta<\xi}C_\delta$ when $\xi$~is a limit, and
at successor stages
$$
C_{\xi+1}=C_\xi\cup
\bigl\{\orpr\beta\eta:
(\exists\varepsilon>\beta)
(\exists F\in\calF_{\beta,\eta})
\bigl(
\bcard{\{\orpr\varepsilon\zeta:F\in\calF_{\varepsilon,\zeta}\}\setminus C_\xi}
     <\kappa
\bigr)
\bigr\}\eqno(\dag)
$$  
This process stops at $\kappa^+$ (or earlier) because $C_{\kappa^++1}=C_{\kappa^+}$.
To see this note that if $\orpr\beta\eta\in C_{\kappa^++1}$,
as witnessed by~$\varepsilon>\beta$ and~$F\in\calF_{\beta,\eta}$, then the
intersection  
$\{\orpr\varepsilon\zeta:F\in\calF_{\varepsilon,\zeta}\}\cap C_{\kappa^+}$
is of cardinality~$\kappa$ and hence already a subset of~$C_\xi$ for some
$\xi<\kappa^+$, which then shows that $\orpr\beta\eta$ is already 
in~$C_{\xi+1}$.
We call $C_{\kappa^+}$ the closure of~$C$ and will denote it by~$\cl{C}$.

\smallskip
If $\orpr\beta\eta\in\cl{C}$ then $G\cap X$~has cardinality~$\kappa$
whenever $G\in\calF_{\beta,\eta}$.
We prove this by induction on~$\xi$.

If $\orpr\beta\eta\in C_0$ then $X\in\calF_{\beta,\eta}$ and we are done.

Going from $\xi$ to~$\xi+1$ 
let $\orpr\beta\eta\in C_{\xi+1}\setminus C_\xi$ and take
$\varepsilon>\beta$ and $F\in\calF_{\beta,\eta}$ such that
$\{\orpr\varepsilon\zeta:F\in\calF_{\varepsilon,\zeta}\}\setminus C_\xi$
has cardinality less than~$\kappa$.
Now if $G\in\calF_{\beta,\eta}$ then 
$\{\orpr\varepsilon\zeta:G\cap F\in\calF{\varepsilon,\zeta}\}$ is a subset
of $\{\orpr\varepsilon\zeta:F\in\calF_{\varepsilon,\zeta}\}$ of 
cardinality~$\kappa$; therefore its intersection~$I$ with~$C_\xi$ has 
cardinality~$\kappa$.
But then, by the inductive assumption, $G\cap F\cap F_{\varepsilon,\zeta}\cap X$
has cardinality~$\kappa$ for all $\orpr\varepsilon\zeta\in I$,
hence certainly $G\cap X$ has cardinality~$\kappa$.

\smallskip
If $\orpr\beta\eta\notin\cl{C}$ then, in particular $\orpr\beta\eta\notin C$,
hence $X\notin\calF_{\beta,\eta}$ and so every member of~$\calF_{\beta,\eta}$
intersects~$\kappa\setminus X$.

Because the ultrafilters are assumed to be uniform this implies that 
$F\cap(\kappa\setminus X)$ has cardinality~$\kappa$ for all~$F$ 
in~$\calF_{\beta,\eta}$.

\smallskip
We let $\calG_{\beta,\eta}$ be the filter generated 
by $\calF_{\beta,\eta}\cup\{X\}$ if $\orpr\beta\eta\in\cl{C}$ and
by~$\calF_{\beta,\eta}\cup\{\kappa\setminus X\}$ otherwise.

\smallskip
We need to show that
$\langle \calG_{\beta,\eta}:\orpr\beta\eta\in\alpha\times\kappa\rangle$
is stratified.

If $\orpr\beta\eta\notin\cl{C}$ then for every $F\in\calF_{\beta,\eta}$
and every $\gamma>\beta$ the intersection
$\{\orpr\gamma\zeta: F\in\calF_{\gamma,\zeta}\}\setminus\cl{C}$
has cardinality~$\kappa$, and so
$\{\orpr\gamma\zeta: F\cap(\kappa\setminus X)\in\calG_{\gamma,\zeta}\}$
has cardinality~$\kappa$ as well.

Next let $\orpr\beta\eta\in\cl{C}$.
We show that for every~$G\in\calF_{\beta,\eta}$ and every $\gamma>\beta$
the set $I_{G,\gamma}=\{\orpr\gamma\zeta\in\cl{C}:G\in\calF_{\gamma,\zeta}\}$ has 
cardinality~$\kappa$. 

Once this is established we see that for $\orpr\gamma\zeta\in G_{F,\gamma}$
we have $F,X\in\calG_{\gamma,\zeta}$ and so hence $F\cap X\in\calG_{\gamma,\zeta}$.
And since $\calG_{\beta,\eta}$ is generated by 
$\{F\cap X:F\in\calF_{\beta,\eta}\}$ this shows that the family is stratified
at each element of~$\cl{C}$.

We prove the statement induction on~$\xi$ that the statement holds for 
every~$C_\xi$.

If $\orpr\beta\eta\in C_0$ then $X\in\calF_{\beta,\eta}$ and so for every 
$G\in\calF_{\beta,\eta}$ and every $\gamma>\beta$ the set 
$\{\orpr\gamma\zeta:G\cap X\in\calF_{\gamma,\zeta}\}$ has cardinality~$\kappa$
and is contained in~$C_0$. 

Going from $\xi$ to $\xi+1$ let $\orpr\beta\eta\in C_{\xi+1}$ as witnessed
by $\varepsilon>\beta$ and $F\in\calF_{\beta,\eta}$.
Let $I=\{\orpr\varepsilon\zeta:F\in\calF_{\varepsilon,\zeta}\}$, 
then $I\setminus C_\xi$ has cardinality less than~$\kappa$.

Now let $G\in\calF_{\beta,\eta}$, then 
$\{\orpr\varepsilon\zeta\in I:G\in\calF_{\varepsilon,\zeta}\}$ has cardinality,
hence so does $I_{G,\varepsilon}\cap C_\xi$.

By the inductive assumption we find that for $\gamma>\varepsilon$ the
set $\{\orpr\gamma\zeta\in C_\xi:G\in\calF_{\gamma,\zeta}\}$ has 
cardinality~$\kappa$.

By the definition of $C_{\xi+1}$ we find that for $\gamma$ in the interval
$(\beta,\varepsilon)$ the set 
$\{\orpr\gamma\zeta:G\cap F\in\calF_{\gamma,\zeta}\}$
is a subset of~$C_{\xi+1}$, as witnessed by~$F$ and $\varepsilon$.
\end{proof}

To create the stratified set of ultrafilters we enumerate the power
set of~$\kappa$ as $\langle X_\nu:\nu\in2^\kappa\rangle$ and recursively
build matrices
$\langle\calF^\nu_{\beta,\eta}:\orpr\beta\eta\in\alpha\times\kappa\rangle$,
where
$\langle\calF^0_{\beta,\eta}:\orpr\beta\eta\in\alpha\times\kappa\rangle$
is the given stratified set,
each time 
$\langle\calF^{\nu+1}_{\beta,\eta}:\orpr\beta\eta\in\alpha\times\kappa\rangle$
is obtained by applying Lemma~\ref{lemma.refining} to
$\langle\calF^\nu_{\beta,\eta}:\orpr\beta\eta\in\alpha\times\kappa\rangle$
and~$X_\nu$,
and at limit stages $\calF^\nu_{\beta,\eta}=\bigcup_{\mu<\nu}\calF^\mu_{\beta,\eta}$
for all~$\orpr\beta\eta$.

Then we can let $u_{\beta,\eta}=\calF^{2^\kappa}_{\beta,\eta}$ 
for all~$\orpr\beta\eta$.

\end{document}